\documentclass{elsarticle}
\usepackage{etex}
\usepackage[utf8x]{inputenc}
\usepackage{amsmath}
\usepackage{pictex}
\usepackage{fixltx2e}
\usepackage{amssymb}
\usepackage{float}
\usepackage{import}
\usepackage{amsthm}
\usepackage{enumitem}
\usepackage[toc,page]{appendix}
\usepackage{listings}
\usepackage{color} 
\usepackage{makeidx}

\makeindex

\newtheorem{theorem}{Theorem}[section]

\newtheorem{proposition}[theorem]{Proposition}
\newtheorem{corollary}[theorem]{Corollary}
\newtheorem{definition}{Definition}
\makeatletter
\def\ps@pprintTitle{%
 \let\@oddhead\@empty
 \let\@evenhead\@empty
 \def\@oddfoot{\centerline{\thepage}}%
 \let\@evenfoot\@oddfoot}
\makeatother

\begin{document}

\begin{frontmatter}
\title{Invariant measures for continued fraction algorithms with finitely many digits}
\author{Cor Kraaikamp}
\address{Delft University of Technology, EWI (DIAM), Mekelweg 4, 2628 CD Delft, the
Netherlands, c.kraaikamp@tudelft.nl} 
\author{Niels Langeveld}
\address{Mathematical Institute, Leiden University, P.O. Box 9512, 2300 RA Leiden, The Netherlands, n.d.s.langeveld.2@umail.leidenuniv.nl} 

\begin{abstract}
\noindent In this paper we consider continued fraction (CF) expansions on intervals different from $[0,1]$. For every $x$ in such interval we find a CF expansion with a finite number of possible digits. Using the natural extension, the density of the invariant measure is obtained in a number of examples.
In case this method does not work, a Gauss-Kuzmin-L\'evy based approximation method is used. Convergence of this method follows from~\cite{SW} but the speed of convergence remains unknown. For a lot of known densities the method gives a very good approximation in a low number of iterations.
Finally, a subfamily of the $N$-expansions is studied. In particular, the entropy as a function of a parameter $\alpha$ is estimated for $N=2$ and $N=36$.
Interesting behavior can be observed from numerical results.
\end{abstract}


\end{frontmatter}

\section{Introduction}

In general, studies on continued fraction expansions focus on expansions for which almost all $x$\footnote{All `almost all $x$' statements are wrt. Lebesgue measure.} have an expansion with digits from an infinite alphabet. A classical example is the regular continued fraction~\cite{DK1,IK,S}.
An example of continued fraction expansions with only finitely many digits has been introduced in~\cite{Le} by Joe Lehner, where the only possible digits are 1 and 2; see also~\cite{DK2}.
More recently, continued fractions have been investigated for which all $x$ in a certain interval have finitely many possible digits.
In~\cite{DKW} the following 4-expansion has been (briefly) studied.
Let $T:[1,2]\rightarrow[1,2]$ be defined as
\begin{equation}\label{4expansiondigit12}
 T(x) = \left\{
  \begin{array}{l l l}
  \dfrac{4}{x}-1 & \text{for} &x\in(\frac{4}{3},2] \\[0.75em]
  \dfrac{4}{x}-2 & \text{for} &x\in[1,\frac{4}{3}] \ . \\
\end{array} \right.
\end{equation}

\begin{figure}[h!]
\import{home/niels/4expansiondigits12.pic}{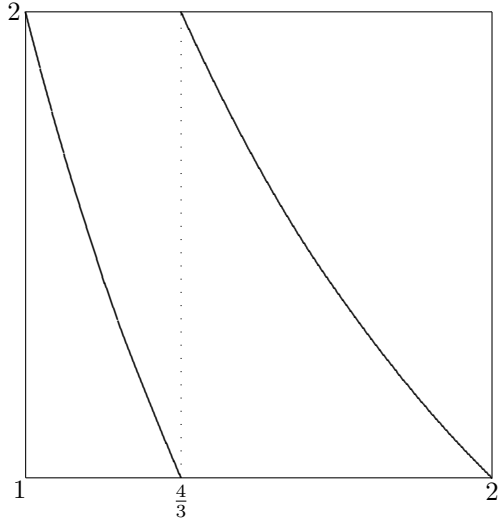}\caption{The CF-map $T$ from (\ref{4expansiondigit12}).}
\end{figure}

\noindent By repeatedly using this map we find that every $x\in[1,2]$ has an infinite continued fraction expansion of the form
\begin{equation}\label{eq:expansionofxdigits12}
 x= \frac{4}{\displaystyle d_1+\frac{4}{\displaystyle d_2+ \ddots }}\\
\end{equation}
\noindent with $d_n\in\{1,2\}$ for all $n\geq1$. The class of continued fractions algorithms that give rise to digits from a finite alphabet is very large. In this paper we will give examples of such expansions and in Section~\ref{sec:subfam} we will take a closer look at an interesting sub-family. Most of the examples will be a particular case of $N$-expansions (see~\cite{AW,BGK,DKW}). Other examples are closely related and can be found by combining the $N$-expansions with flipped expansions (cf.~\cite{L14} for 2-expansions; see also~\cite{DHMK} for flipped expansions).
For all these examples we refer to~\cite{DHMK} for ergodicity (which can be obtained in all these cases in a similar way) and existence of an invariant measure.
In a number of cases however, it is difficult to find the invariant measure explicitly, while in seemingly closely related cases it is very easy.
In case we cannot give an analytic expression for the invariant measure we will give an approximation using a method that is very suitable (from a computational point of view) for expansions with finitely many different digits.
This method is based on a Gauss-Kuzmin-L\'evy theorem. For greedy $N$-expansions this theorem is proved by Dan Lascu in~\cite{La}. The method yields smoother results than by simulating in the classical way (looking at the histogram of the orbit of a typical point as described in Choe's book~\cite{C3}, and used in his papers~\cite{C1,C2}).
We also give an example in which we do know the density and where we use this method to show its strength.
\vspace{1em}

\noindent In Section~\ref{sec:general} we will give the general form of the continued fraction maps we study in this paper. After that we give several examples of such maps and a way of finding the density of the invariant measure by using the natural extension.
In Section~\ref{sec:sim} we will see how we simulated the  densities in case we were not able to find them explicitly. In the last section we will consider a subfamily of the $N$-expansions which can be parameterized by $\alpha\in(0,\sqrt{N}-1]$.
We study the entropy as function of $\alpha$. We will do so mainly on a numerical basis. In the past decades, it turned out to be very interesting to look at entropy of a family of continued fractions as a function of a parameter. In~\cite{BCIT,CMPT,CT12,CT13,KSS2,LM,NN} the entropy of Nakada's $\alpha$-expansions is studied. For example in~\cite{NN} it is shown that in any neighborhood of 0 you can always find an interval on which the entropy function is increasing, an interval on which the entropy function is decreasing and an interval on which this function is constant.
In~\cite{CT12} it is shown that there is a countable set of open intervals on which the entropy function is monotonic. The union of these intervals has Lebesgue measure 1. As matching plays an important role in these papers we will take a glimpse at how matching works for our subfamily (see Section~\ref{sec:subfam} for the definition and use of matching).

\section{The general form of our maps}\label{sec:general}
Throughout this paper we will look at continued fraction algorithms of the following form.
Fix an integer $N\geq2$ and let $[a,b]$ be a subinterval of $[0,N]$ with $b-a\geq 1$. Let $T:[a,b]\rightarrow[a,b]$ be defined as
$$
T(x)=\frac{\varepsilon(x)N}{x}-\varepsilon(x)d(x)
$$
where $\varepsilon(x)$ is either $-1$ or $1$ depending on $x$ and $d(x)$ is a positive integer such that $T(x)\in[a,b]$. Note that if $b-a=1$ then there is exactly one positive integer such that $T(x)\in[a,b)$ if $\varepsilon(x)$ is fixed. For $N=2$ we find the family that is studied in~\cite{L14} and for $\varepsilon(x)=1$ for all $x$ we find the $N$-expansions from~\cite{DKW}.
Whenever $a>0$ this map can only have finitely many different digits. This family is closely related to the  $(a,b)$-continued fractions introduced and studied by Svetlana Katok and Ilie Ugarcovici in~\cite{KU0,KU1,KU2}. For $(a,b)$-continued fraction $\varepsilon(x)=-1$ for all $x\in[a,b]$ and  $N=1$. Also there are restrictions on $a,b$. These are chosen such that $a\leq0\leq b, \  b-a\geq1, \  -ab\leq1$.
The case $b-a=1$ was further studied by Carlo Carminati, Stefano Isola and Giulio Tiozzo in~\cite{CIT}, where they study the entropy as a function of $\alpha$.

\vspace{1em}

\noindent Note that this family is rather `large'. For the examples in the next section $\varepsilon(x)$ will be plus or minus one on fixed interval(s). In Section~\ref{sec:subfam} other restrictions are imposed.
\subsection{Two seemingly closely related examples and their natural extension}

In~\cite{DKW}, using the natural extension the invariant measure of the 4-expansion map $T$ given in (\ref{4expansiondigit12}) was easily obtained. To (briefly) illustrate the method and the kind of continued fraction algorithms we are interested in we consider a slight variation of this continued fraction.
Let $\tilde{T}:[1,2]\rightarrow[1,2]$ be defined as
\begin{equation}\label{4expansiondigits15}
 \tilde{T}(x) = \left\{
  \begin{array}{l l l}
  \dfrac{4}{x}-1 & \text{for} & x\in(\frac{4}{3},2] \\[0.75em]
  5-\dfrac{4}{x} & \text{for} & x\in[1,\frac{4}{3}] \ ,  \\
\end{array} \right.
\end{equation}
i.e. we `flipped' the map $T$ on the interval $[1,\frac{4}{3}]$; see also Figure~\ref{fig:4expdigit15}.
\begin{figure}[h!]
\import{home/niels/4expansiondigits15.pic}{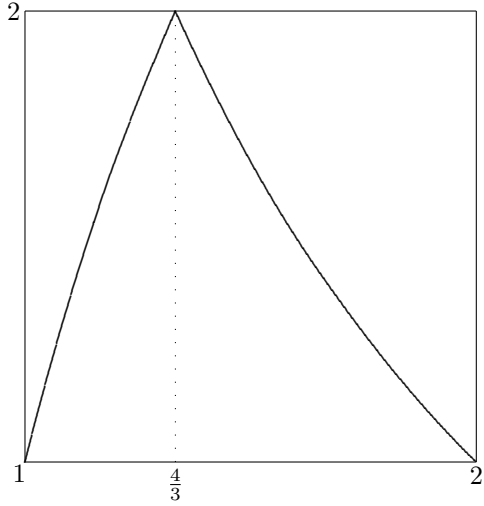}\caption{The CF map $\tilde{T}$ from (\ref{4expansiondigits15}). }\label{fig:4expdigit15}
\end{figure}
\newline Setting
$$
 \varepsilon_1(x) = \left\{
  \begin{array}{l l}
  1 & \text{for $x\in (\frac{4}{3},2]$}\\
 -1 & \text{for $x\in [1,\frac{4}{3}]$}\\
  \end{array} \right.
\quad \text{and } \quad d_1(x) = \left\{
  \begin{array}{l l}
    1 & \quad \text{for $x\in(\frac{4}{3},2]$}\\
    5 & \quad \text{for $x\in [1,\frac{4}{3}]$}\ ,\\
  \end{array} \right.
$$
we define $\varepsilon_n(x)=\varepsilon_1\left(T^{n-1}(x)\right)$ and $d_n(x)=d_1\left(T^{n-1}(x)\right)$. From $T(x)=\varepsilon_1\cdot\left(\frac{4}{x}-d_1\right)$, it follows that

\begin{equation}
 x = \frac{4}{d_1+\varepsilon_1 T(x)} =\ldots    = \frac{4}{\displaystyle d_1+\frac{4 \,\varepsilon_1}{\displaystyle d_2+ \ddots + \frac{4 \,\varepsilon_{n-1}}{\displaystyle d_n +\varepsilon_n T^n(x)}}} \ .
\end{equation}

\noindent Taking finite truncations, we find the so called convergents
$$
c_n=\frac{p_n}{q_n}=\frac{4}{\displaystyle d_1+\frac{4 \,\varepsilon_1}{\displaystyle d_2+ \ddots + \frac{4 \,\varepsilon_{n-1}}{\displaystyle d_n}}}
$$
of $x$. One can show that $\lim_{n\rightarrow\infty}c_n=x$; see~\cite{L14} for further details. Therefore we write
\begin{equation}\label{eq:expansionofx}
 x= \frac{4}{\displaystyle d_1+\frac{4 \,\varepsilon_1}{\displaystyle d_2+ \ddots }}\\ \ ,
\end{equation}
or in short hand notation $x=[4/d_1,4\varepsilon_1/d_2,\ldots]$ or $x=[d_1,\varepsilon_1/d_2,\ldots]_4$.

\subsubsection{Using the natural extension to find the invariant measure}\label{sec:natex}
To demonstrate the method we will use $\tilde{T}$ from (\ref{4expansiondigits15}).
The idea now is to build a two-dimensional system (the natural extension) $\left(\Omega=[1,2]\times[A,B],\mathcal{T}\right)$ which is almost surely invertible and contains $([1,2],\tilde{T})$ as a factor.
In~\cite{DKW} it was shown that a suitable candidate for the natural extension map $\mathcal{T}$ is given by
$$
\mathcal{T}(x,y)=\left(\tilde{T}(x),\frac{4\varepsilon_1(x)}{d_1(x)+y}\right) \ .
$$
\begin{figure}[h!]
\import{home/niels/natexN4digits15.pic}{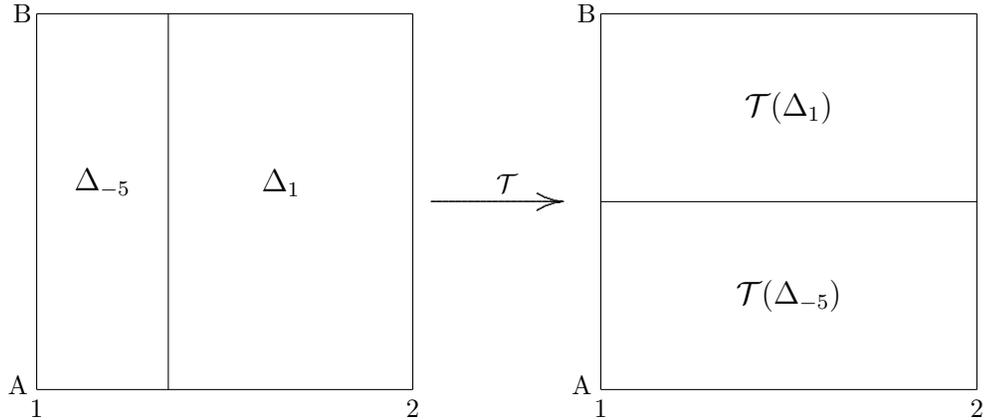}\caption{The suitable domain for $\mathcal{T}$.}\label{fig:4natexpdigit15}
\end{figure}
\noindent Now we choose $A$ and $B$ in such a way the system is indeed (almost surely) invertible. We define fundamental intervals $\Delta_{n}=\{(x,y)\in\Omega :d_1(x)=n\}$ if $\varepsilon=1$ and $\Delta_{-n}=\{(x,y)\in\Omega :d_1(x)=n\}$ if $\varepsilon=-1$.
When the fundamental intervals fit exactly under the action of $\mathcal{T}$, the system is almost surely invertible; see Figure~\ref{fig:4natexpdigit15}. An easy calculation shows that $A=-1$ and $B=\infty$ is the right choice here.
It is shown in~\cite{DKW}, that the density of the invariant measure (for the $2$-dimensional system) is given by
$$
f(x,y)=C\frac{4}{(4+xy)^2} , \quad \text{for }  (x,y)\in\Omega \ ,
$$
where $C$ is a normalizing constant (which is $\frac{1}{\log(3)}$ in this example). Projecting on the first coordinate yields the invariant measure for the $1$-dimensional system $([1,2],\tilde{T})$, with density
$$
\frac{1}{\log(3)}\left(\frac{1}{x}+\frac{1}{4-x}\right) , \quad \text{for } x\in[1,2] \ .
$$

\noindent Note that if we would consider the map
\begin{equation}\label{4expansiondigits24}
 \hat{T}(x) = \left\{
  \begin{array}{l l l}
  4-\dfrac{4}{x} & \text{for }  x\in(\frac{4}{3},2] \\[0.75em]
  \dfrac{4}{x}-2 & \text{for }  x\in[1,\frac{4}{3}] \ , \\
\end{array} \right.
\end{equation}

\begin{figure}[h!]
\import{home/niels/4expansiondigits24.pic}{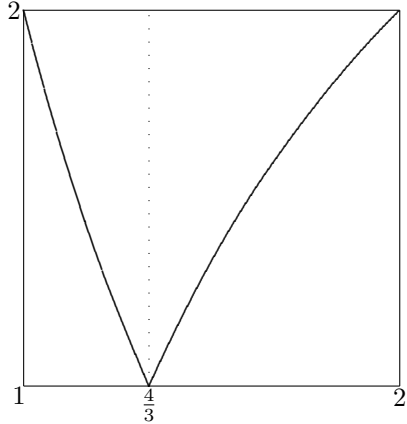}\caption{The CF map $\hat{T}$ from (\ref{4expansiondigits24}).}\label{fig:4expansiondigits24}
\end{figure}
\noindent which is a `flipped version' of the map $T$ from (\ref{4expansiondigit12}) where we flipped the branch on the interval $(\frac{4}{3},2]$, we get another continued fraction of the form (\ref{eq:expansionofx}) but now with digits $d_n\in\{2,4\}$; see Figure~\ref{fig:4expansiondigits24}. Our approach would give $A=-2$ and $B=\infty$ which shows that the underlying dynamical system has a $\sigma$-finite infinite measure with `density' $f(x)$, given by
$$
f(x)=\frac{1}{x}+\frac{1}{2-x}, \quad \text{for } x\in[1,2].
$$

\noindent The method from~\cite{DKW} we just used does not always `work'. As an example we will use an expansion given in~\cite{L14}.
Let $\bar{T}(x)$ be defined as
\begin{equation}\label{eq:tiny}
 \bar{T}(x) = \left\{
  \begin{array}{l l l}
  \dfrac{2}{x}-3 & \text{for} & x\in(\frac{1}{2},\frac{4}{7}] \\[0.75em]
  4-\dfrac{2}{x} & \text{for} & x\in(\frac{4}{7},\frac{2}{3}] \\[0.75em]
  \dfrac{2}{x}-2 & \text{for} & x\in(\frac{2}{3},\frac{4}{5}] \\[0.75em]
  3-\dfrac{2}{x} & \text{for} & x\in[\frac{4}{5},1] \ , \\
\end{array} \right.
\end{equation}
see Figure~\ref{fig:atractor}.

\begin{figure}[h!]
\import{home/niels/atractor.pic}{atractor.pic}
\end{figure}

\noindent When trying to construct the domain of the natural extension one quickly notices that `holes' appear. This is not an entirely new phenomena in continued fractions, it also appears in constructing the natural extension of Nakada's $\alpha$-expansions when $\alpha\in(0,\sqrt{2}-1)$; see~\cite{LM}.
One might hope that there are finitely many holes, but a simulation of the domain indicates otherwise; see Figure~\ref{fig:atractor12to1domain}.
\begin{figure}[h!]
\includegraphics[trim= 1cm 10cm 4cm 14cm, width=95mm]{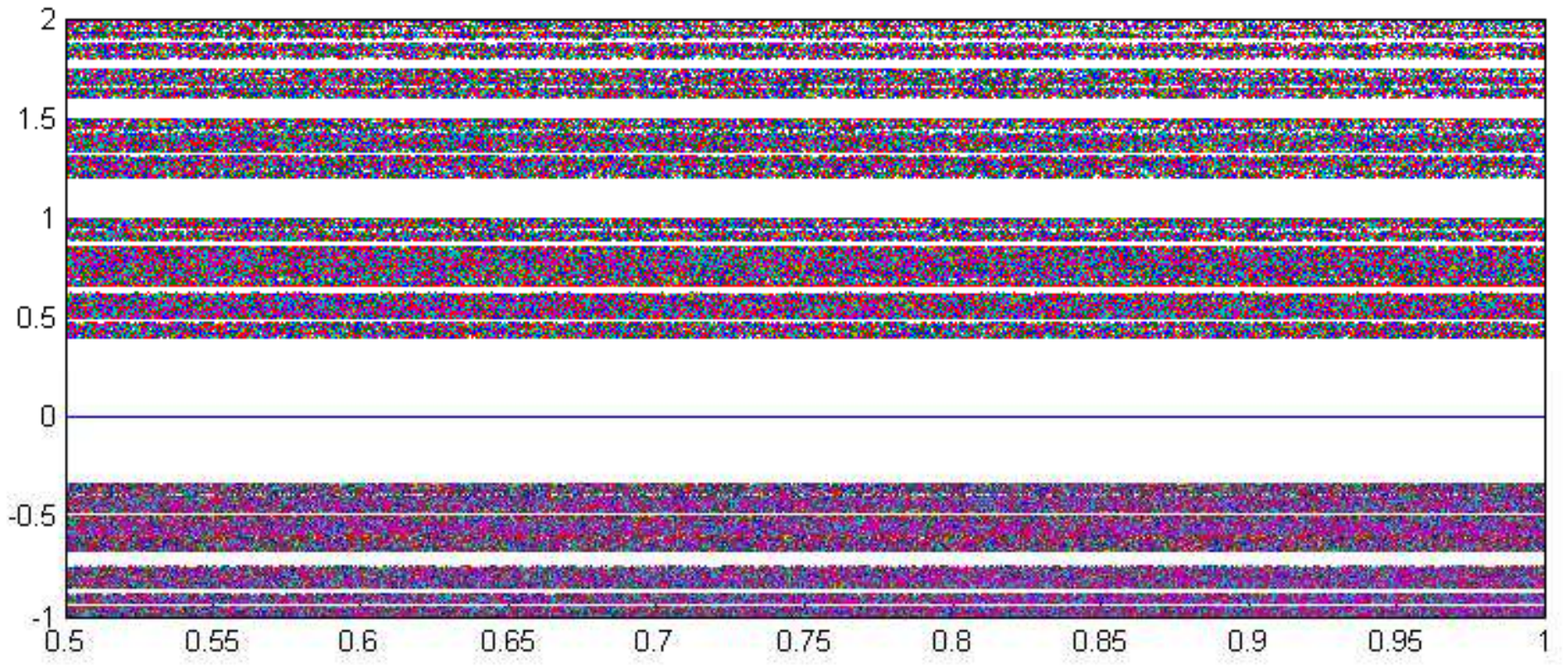}\caption{A simulation of the domain of the natural extension for the map $\bar{T}$ from (\ref{eq:tiny}).}\label{fig:atractor12to1domain}
\end{figure}

\noindent Although the method might still work in this case, it does not really seem to help us to find a description of the invariant density. In order to get an idea of the density we will use two different approaches. One will be based on the Gauss-Kuzmin-L\'evy Theorem. The other will be a more classical approach based on Choe's book~\cite{C3}.

\subsection{Two different ways of approximating the density}\label{sec:sim}
The first way is based on the Gauss-Kuzmin-L\'evy theorem. This theorem states that for the regular continued fraction the Lebesgue measure of the pre-images of a measurable set $A$ will converge to the Gauss measure.
$$
\lambda\left(T^{-n}(A)\right)\rightarrow\mu(A) \quad \text{as } n\rightarrow\infty.
$$
This was stated as an hypothesis by Gauss in his mathematical diary in 1800 and proved by Kuzmin in 1928 who also obtained a bound on the speed of convergence. Independently, L\'evy proved the same theorem in 1929 but found a sharper bound for the speed of convergence namely
$|\lambda\left(T^{-n}(A)\right)-\mu(A)|=\mathcal{O}(q^n)$ with $0<q<1$ instead of $\mathcal{O}(q^{\sqrt{n}})$ which is the bound Kuzmin found.
There are many proofs of this theorem and refinements on the speed of convergence; see e.g. Khinchine's book~\cite{IK}, or~\cite{K} for various refinements.
\newline The idea for our method is to look at the pre-images of $[\frac{1}{2},z]$ for our map $\bar{T}$ from (\ref{eq:tiny}) and take the Lebesgue measure of the intervals found. Note that in the first iteration you will find 2 intervals and after the first iteration the number of intervals found is multiplied by 4.
Also the size of the intervals shrink relatively fast. Fortunately it seems that a low number of iterates (around 10) is already enough to give a good approximation; see Figure~\ref{fig:gausskuzdigits12} where the theoretical density and its approximation are displayed and
 figure~\ref{fig:densatractorplusgk}, where both methods of approximating are compared for a density we do not know the theoretical density.

\vspace{1em}

\noindent The other way of finding an approximation is by iterating points and look at the histogram of the orbits. The way we iterated is that we used a lot of points and iterated them just a few times. To be more precise we iterated 2500 uniformly random points 20 times, repeated this process 400 times and took the average density of all points. Then we redid the process but instead of sampling uniformly we sampled from the previously found density (see also~\cite{DKL})
In Figure~\ref{fig:densatractorplusgk} we see both methods applied to our example.

\begin{figure}[h!]
\includegraphics[trim= 3cm 0cm 0cm 0cm, width=120mm]{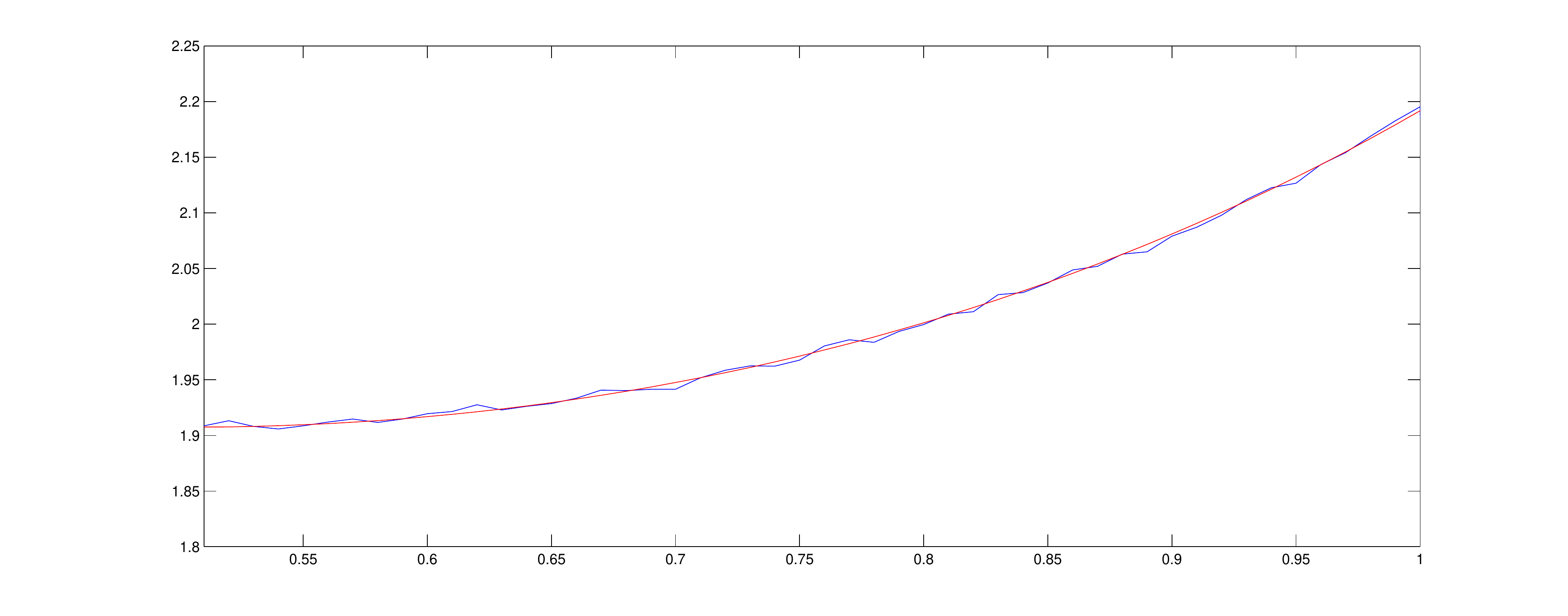}\caption{Approximations of the density of the invariant measure of $\bar{T}(x)$ using the Gauss-Kuzmin-L\'evy method (red) and the classical way (blue).}\label{fig:densatractorplusgk}
\end{figure}

\noindent The two methods give results that are relatively close but the approximation found with the Gauss-Kuzmin-L\'evy method is far more smooth. Since we do not know the density we cannot compare the theoretical density with the approximation. Since the Gauss-Kuzmin-L\'evy method is the new method we will look how well this method performs in an example in which we know the invariant density explicitly.
For the map $T$ from (\ref{4expansiondigit12}) we know the density which was given in~\cite{DKW}. In Figure~\ref{fig:gausskuzdigits12} we see a plot of both the theoretical density and the approximation found by the Gauss-Kuzmin-L\'evy method.
\begin{figure}[h!]
\includegraphics[trim =3cm 0cm 0cm 0cm, width=125mm]{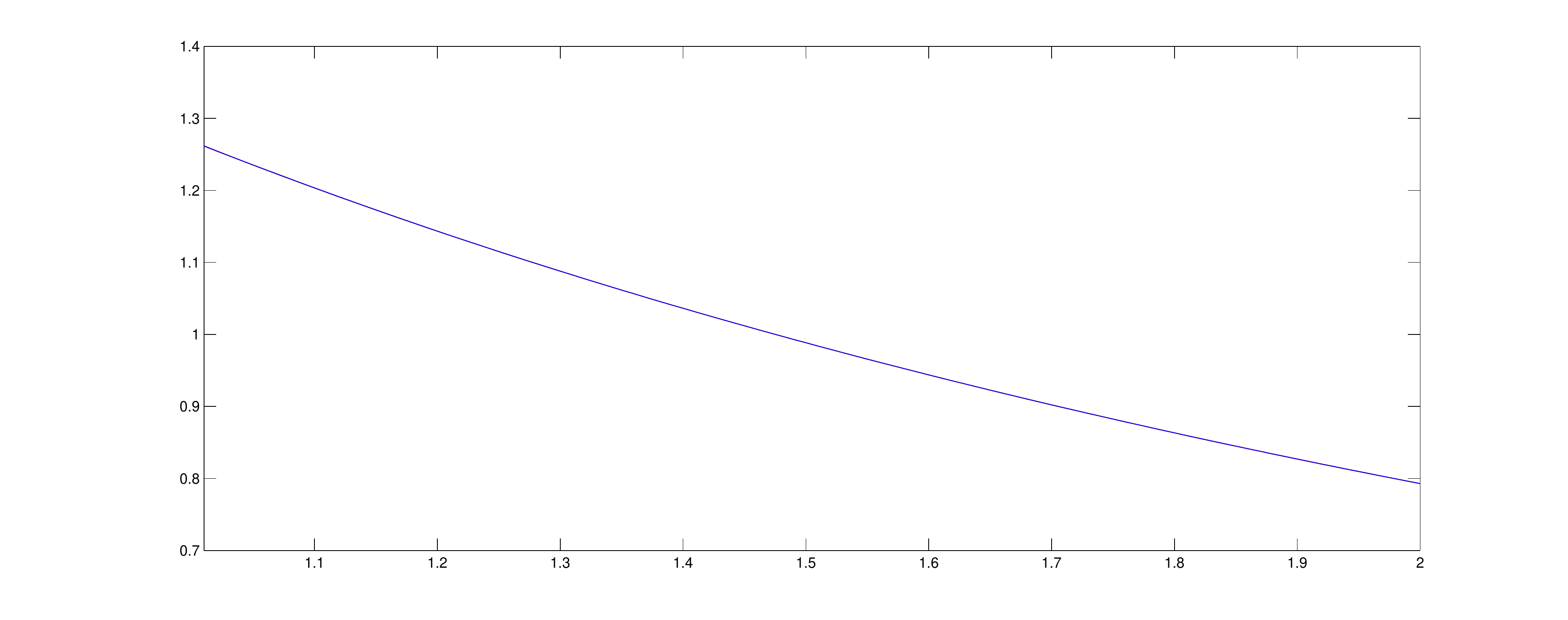}\caption{An approximation of the true density for the CF map $T$ from (\ref{4expansiondigit12}) and the true $T$-invariant density.}\label{fig:gausskuzdigits12}
\end{figure}

\noindent The difference can barely be seen by the naked eye. If we look at the difference in 2-norm we get
 $$\left(\int_1^2 (f(x)-\hat{f}(x))^2 \, dx\right)^{\frac{1}{2}}=1.1235*10^{-5}$$
where $f(x)$ is the true density and $\hat{f}(x)$ the approximation.

\section{A sub-family of the $N$-expansions}\label{sec:subfam}
In this section we study a subfamily of the $N$-expansions (so $\varepsilon(x)=1$ for all $x$ in the domain) with digits from a finite alphabet and an interval $[\alpha,\beta]$ as domain. For our subfamily we want that it has finitely many digits. Furthermore we would like that there is a unique digit such that $T(x)\in[\alpha,\beta)$. This results in $\alpha>0$ and $\beta-\alpha=1$.
In this way the map is uniquely determined by the domain. With these restrictions we define for $\alpha\in(0,\sqrt{N}-1]$ the map $T_{\alpha,N}:[\alpha,\alpha+1]\rightarrow[\alpha,\alpha+1]$ as
$$
T_{\alpha,N}=\frac{N}{x}-\left\lfloor \frac{N}{x}-\alpha\right\rfloor \ .
$$
Note that for all these expansions we have a finite number of digits since $\alpha>0$. Also note that this is the largest range in which we can choose $\alpha$ because for $\alpha>\sqrt{N}-1$ the digit would be 0 or less (see Section~\ref{sec:admis} for a calculation).
Simulations show that a lot of maps have an attractor smaller than $[\alpha,\beta]$. When $N\geq9$ we find that if $\alpha=\sqrt{N}-1$ there is always an interval $[c,d]\subsetneq[\alpha,\alpha+1]$ for which the $T_{\alpha,N}$-invariant measure of $[c,d]$ is zero. Whenever $N>4$ we have that
$T_{\alpha,N}$ always has 2 branches for $\alpha=\sqrt{N}-1$. Calculations of these observations are given in Section~\ref{sec:admis} in which we take a closer look on which sequences are admissible for a given $N$ and $\alpha$. In Section~\ref{sec:entropy} we study the behavior of the entropy as a function of $\alpha$ for a fixed $N$.
\vspace{1em}

\noindent The examples in~\cite{DKW} with `fixed range' are all member of this kind of sub-family of the $N$-expansions. Though, these examples are cases for which all the branches of the mapping are full. In such case the natural extension can be easily build using the method described previously.
If not all branches are full we can still make the natural extension in some cases. We will start this section with such case.

\subsection{A 2-expansion with $\alpha=\sqrt{2}-1$}\label{sec:N2asqrt2min1}

Let $T(x):[\sqrt{2}-1,\sqrt{2}]\rightarrow[\sqrt{2}-1,\sqrt{2}]$ be defined by
$$
 T(x) = \left\{
  \begin{array}{l l l }
  \dfrac{2}{x}-1 & \text{for} & 2(\sqrt{2}-1) <  x \leq \sqrt{2} \\[0.75em]
  \dfrac{2}{x}-2 & \text{for} & 2-\sqrt{2}  <  x \leq 2(\sqrt{2}-1) \\[0.75em]
  \dfrac{2}{x}-3 & \text{for} & \frac{1}{7}(6-2\sqrt{2})<  x \leq 2-\sqrt{2}    \\[0.75em]
  \dfrac{2}{x}-4 & \text{for} & \sqrt{2}-1 \leq  x \leq \frac{1}{7}(6-2\sqrt{2}) \ . \\
  \end{array} \right.
$$

\begin{figure}[h!]
\import{home/niels/2expansionatractor.pic}{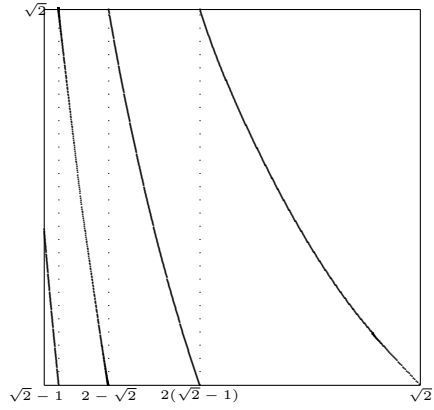}\caption[triangle]{A 2-expansion on the interval $[\sqrt{2}-1,\sqrt{2}]$.} \label{fig:2expansionatractor1}
\end{figure}

\noindent A graph of this map is shown in Figure~\ref{fig:2expansionatractor1}. We can find the invariant measure for this map by using the method as in Section~\ref{sec:natex} though we now need to determine 3 `heights' in order to make the mapping of the natural extension almost surely bijective on the domain (see Figure~\ref{fig:nnatex2expansionatractor}).
We get the following equations for the heights $A, B$ and $C$:
$$A=\frac{2}{4+C} \  , \quad B=\frac{2}{3+C} \quad \text{and } C=\frac{2}{1+B} \ .$$
\begin{figure}[h]
\import{home/niels/natex2expansionatractor.pic}{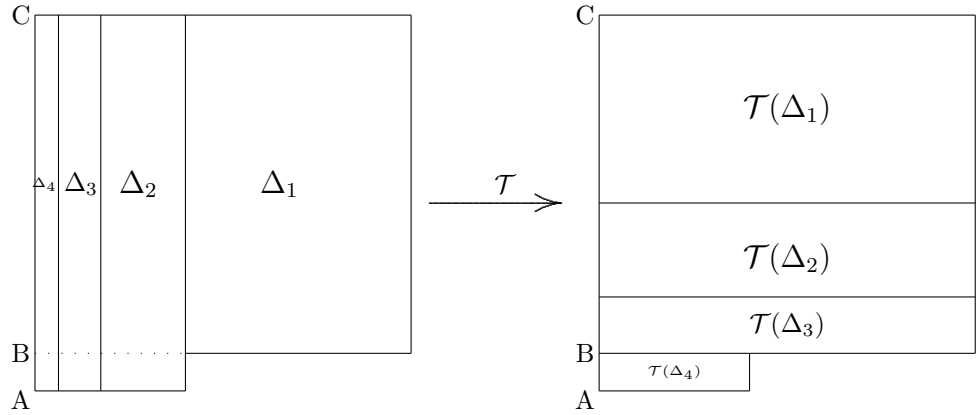}\caption{$\Omega$ and $\mathcal{T}(\Omega)$.}\label{fig:nnatex2expansionatractor}
\end{figure}
\newpage
\noindent This results in $A=\frac{1}{2}(\sqrt{33}-5)$, $B=\frac{1}{6}(\sqrt{33}-3)$ and $C=\frac{1}{2}(\sqrt{33}-3)$.
We find the following invariant density up to a normalizing constant (which is given in Theorem~\ref{thm:natexonint})
$$
f(x)= \left\{
  \begin{array}{l l}
   \frac{\sqrt{33}-3}{4+(\sqrt{33}-3)x}-\frac{\sqrt{33}-5}{4+(\sqrt{33}-5)x} & \text{for } \sqrt{2}-1 <  x \leq 2(\sqrt{2}-1) \\[0.75em]
   \frac{\sqrt{33}-3}{4+(\sqrt{33}-3)x}-\frac{\sqrt{33}-3}{12+(\sqrt{33}-3)x} & \text{for } 2(\sqrt{2}-1) <  x \leq \sqrt{2} \ . \\
  \end{array}\right.
$$
\noindent The graph of the density is given in Figure~\ref{fig:dens2expansionpart1}.
\begin{figure}[h!]
\includegraphics[trim= 1cm 8.5cm 1cm 8cm, width=95mm]{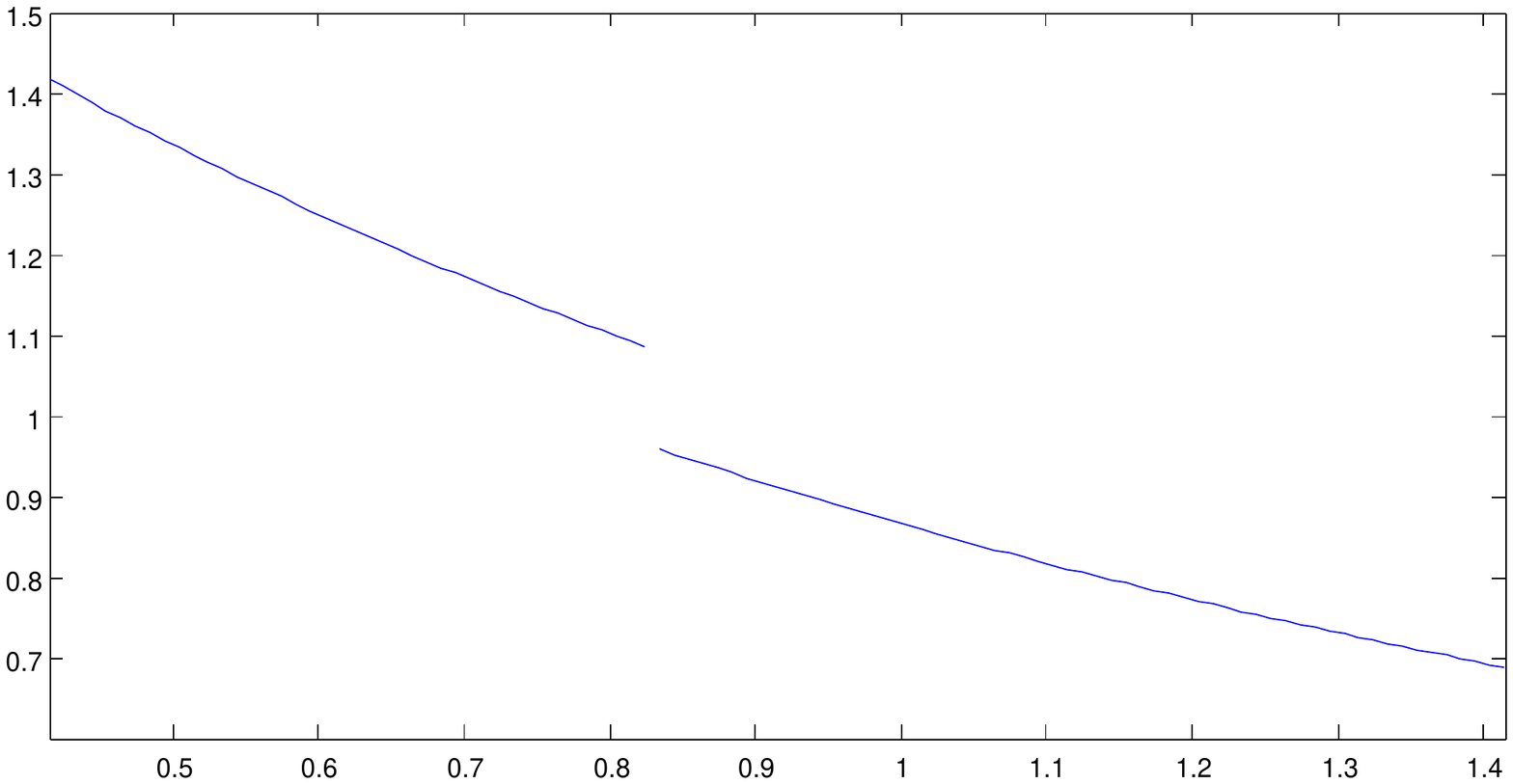}\caption{The density of the invariant measure for the 2-expansion on $[\sqrt{2}-1,\sqrt{2}]$.}\label{fig:dens2expansionpart1}
\end{figure}
\noindent In this case we were lucky. But in general it seems to be very hard to construct the natural extension explicitly. Still we can simulate the densities and calculate the entropy for a given $\alpha$. Also for the 2-expansions we can extend the above result to all $\alpha\in [\frac{\sqrt{33}-5}{2},\sqrt{2}-1]$; see Theorem~\ref{thm:matching on last interval}

\subsection{Admissibility}\label{sec:admis}
In this section we look at how the alphabet is determined by $\alpha$ for a fixed $N$. It turns out that not all different sequences of such alphabet will occur in a continued fraction expansion (or only finitely many times). This is a consequence of some cylinders having zero mass.
The range of the first digits of a continued fraction for given $\alpha$ and $N$ are easily described since the smallest digit will be attained by the right end point of the domain largest digits will be attained by the left end point of the domain. Let
$$n_{min}=\left\lfloor \frac{N}{\alpha+1}-\alpha\right\rfloor \quad \text{and} \quad n_{max}=\left\lfloor \frac{N}{\alpha}-\alpha\right\rfloor \ .$$
Note that $n_{min}\leq 0$ when $\alpha>\sqrt{N}-1$ and therefore $\alpha=\sqrt{N}-1$ is the largest value for which we have positive digits. Furthermore the alphabet is given by $\{n_{min},\ldots,n_{max}\}$.
Now to see for which $N$ we have that for $\alpha=\sqrt{N}-1$ there are two branches we must check that $n_{max}=2$. This happens when $\frac{N}{\sqrt{N}-1}-2\in[\sqrt{N}-1,\sqrt{N}]$. We find 2 inequalities
\begin{equation}\label{eq:2branch1}
 \frac{N}{\sqrt{N}-1}-2\leq \sqrt{N}
\end{equation}
and
\begin{equation}\label{eq:2branch2}
 \sqrt{N}-1\leq\frac{N}{\sqrt{N}-1}-2 \ .
\end{equation}
Inequality (\ref{eq:2branch1}) gives $4\leq N$ and inequality (\ref{eq:2branch2}) gives $N-1\leq N$. This yields that for all $N\geq 4$ we have two branches for $\alpha=\sqrt{N}-1$.
If $N\geq 9$ we also have an attractor which is strictly smaller than the entire interval for $\alpha=\sqrt{N}-1$ as the following calculation shows
\begin{eqnarray*}
 T_{\alpha,N}\left([\alpha,\frac{N}{\alpha+2}]\right)&=&\left[\alpha,\frac{N}{\alpha}-2\right]\\
 T_{\alpha,N}\left([\alpha,\frac{N}{\alpha}-2]\right)&=&\left[\alpha,\frac{N}{\alpha}-2]\cup[\frac{N\alpha}{N-2\alpha}-1,\alpha+1\right]\\
 T_{\alpha,N}\left([\frac{N\alpha}{N-2\alpha}-1,\alpha+1]\right)&=&\left[\alpha,\frac{N^2-(1-3\alpha)N-2\alpha}{(\alpha-1)N+2\alpha}\right] \ .\\
\end{eqnarray*}
If we substitute $\alpha$ with $\sqrt{N}-1$ and the following two inequalities hold we find an attractor strict smaller than the interval $[\alpha,\alpha+1]$;
\begin{equation}
 \frac{N}{\sqrt{N}-1}-2<\frac{N(\sqrt{N}-1)}{N-2(\sqrt{N}-1)}-1
\end{equation}
\begin{equation}
\frac{N^2-(4-3\sqrt{N})N-2(\sqrt{N}-1)}{(\sqrt{N}-2)N+2(\sqrt{N}-1)} <\frac{N}{\sqrt{N-1}}-2 \ ,
\end{equation}
yielding that $N\geq9$.
We take a closer look at $N=9$ and $\alpha=\sqrt{9}-1=2$. This example is briefly discussed in~\cite{DKW} where it is stated that computer experiments suggest that the orbit of 2 never becomes periodic and therefore it is hard to find the natural extension explicitly.
However, when simulating the natural extension, it seems that there are finitely many discontinuities; see Figure \ref{pic:natexN9a2}.
\begin{figure}[h!]
\includegraphics[trim= 1cm 9cm 0cm 8cm, width=100mm]{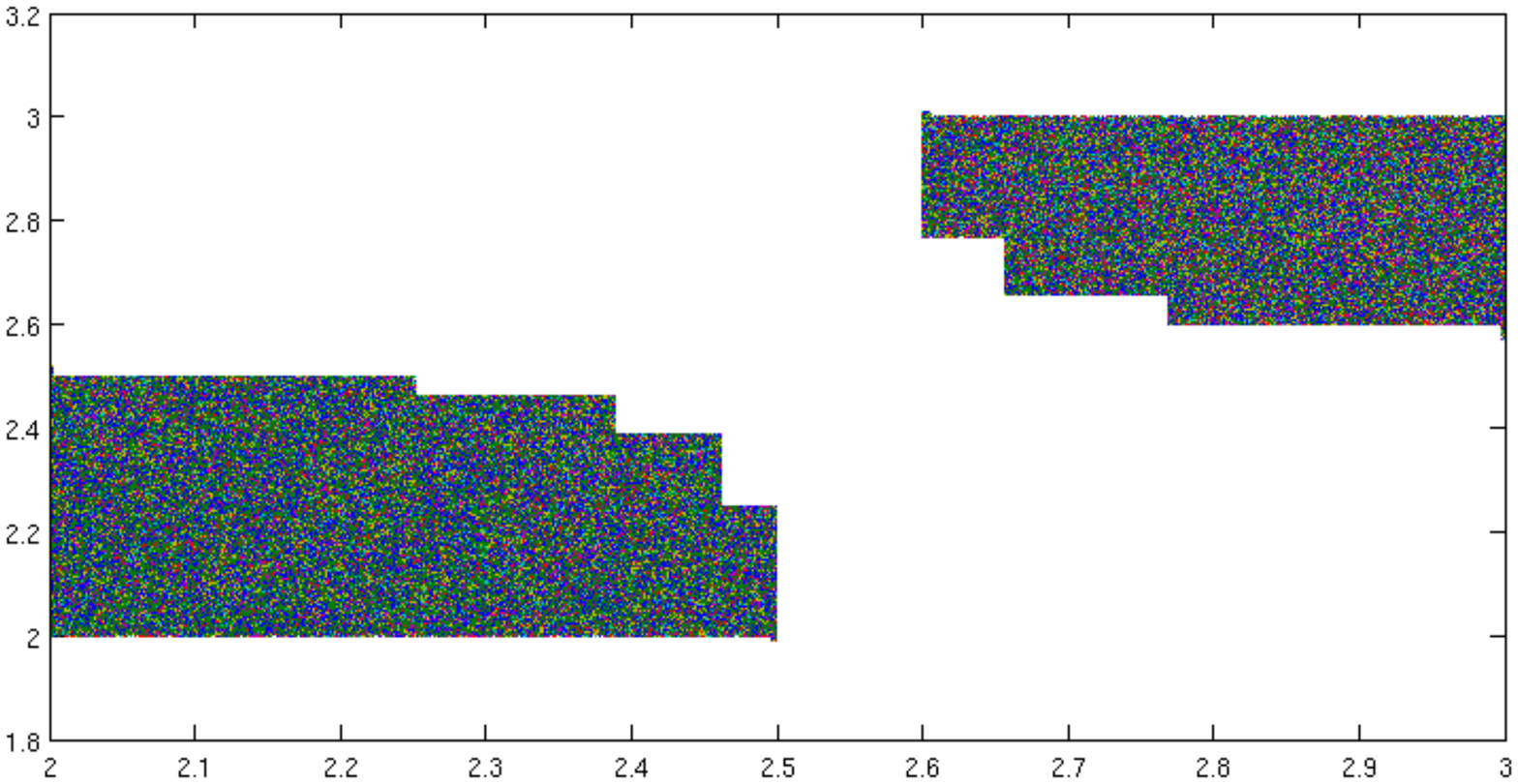}\caption{A simulation of the natural extension for $N=9$ and $\alpha=2$.}\label{pic:natexN9a2}
\end{figure}

We can also simulate the density of the invariant measure; see Figure~\ref{pic:densN9a2}.
\begin{figure}[h!]
\includegraphics[trim= 1cm 8cm 1cm 7cm, width=90mm]{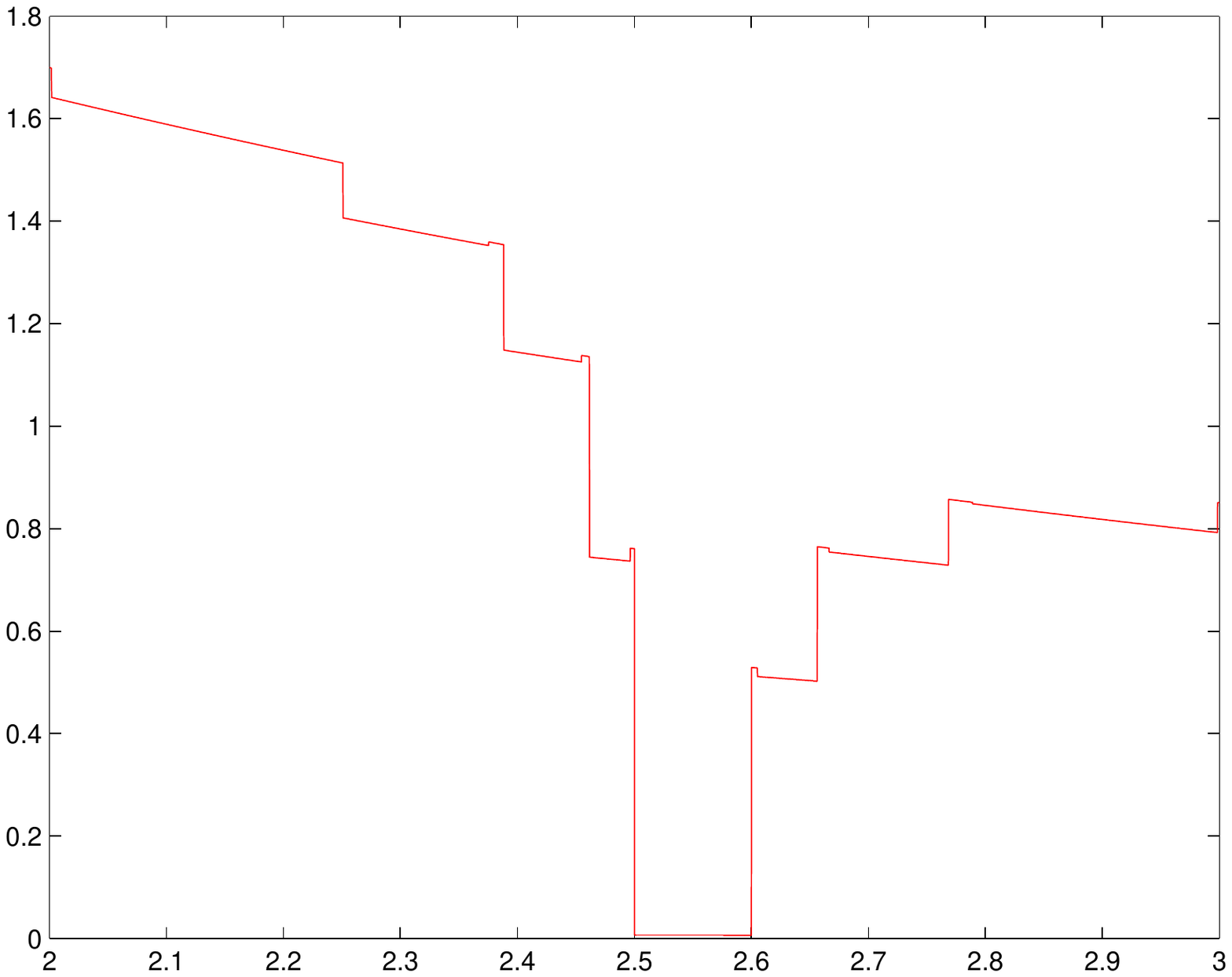}\caption{A Simulation of the density for $N=9$ and $\alpha=2$ using the Gauss-Kuzmin-L\'evy method.}\label{pic:densN9a2}
\end{figure}

\noindent Remark that cylinders with zero mass tells us which sequences are not apparent in any continued fraction of numbers outside the attractor and for those numbers not in the attractor these sequences only appear in the start of the continued fraction.
We can describe which cylinders these are. The hole is given by $[2.5,2.6]$. Now $2.5=[1,1,1,1,1,1,1,1,2,1,1,1,1,1,1,2,2,1,1,1,\ldots]_9$ and $2.6=[1,1,1,1,1,1,1,2,1,1,1,1,1,1,2,2,1,1,1,1\ldots]_9$.
The boundary of a cylinder $\Delta(a_1,\ldots,a_n)$ is given by $[a_1,\ldots,a_n,1,r]_9$ and $[a_1,\ldots,a_n,2,r]_9$ where $r$ is the expansion of $2$. Now a cylinder is contained in $[2.5,2.6]$ if
$2.5<[a_1,\ldots,a_n,1,r]_9<2.6$ and $2.5<[a_1,\ldots,a_n,2,r]_9<2.6$. Note that here we can have a clear description of the attractor and therefore for the admissible sequences. Simulation shows us that there are a lot of different settings in which you find an attractor strictly smaller than the interval.
In Figure~\ref{pic:atractors} simulations for several values of $N$ are shown. On the $y$-axes $\alpha$ is given and on the $x$-axes the attractor is plotted. For example, for $N=9$ we see that for $\alpha=1$ there is no attractor strictly smaller than the interval. There is an attractor for $\alpha=1.8$ and also for example for $\alpha=2$.
The pattern seems to be rather regular. Moreover, more `holes' seem to appear for large $N$ and large $\alpha$.
\begin{figure}[h!]
\centering
\includegraphics[width=55mm]{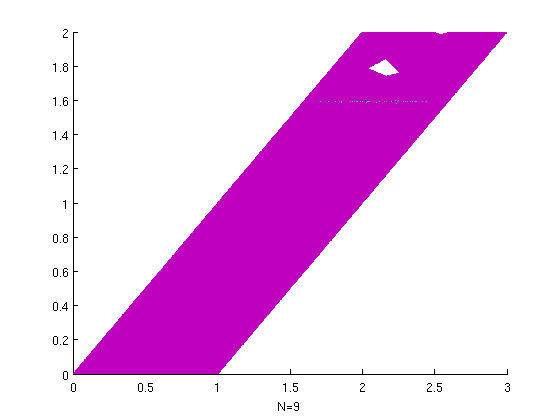}
\includegraphics[width=55mm]{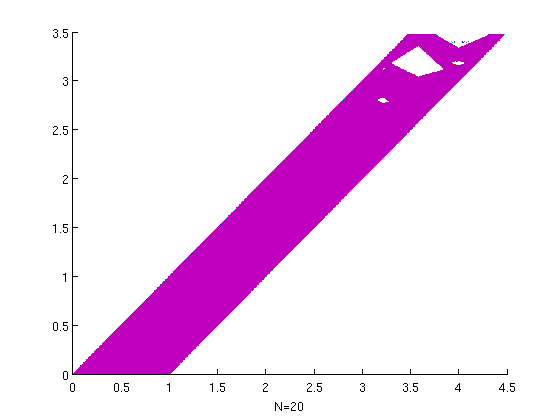}
\includegraphics[width=55mm]{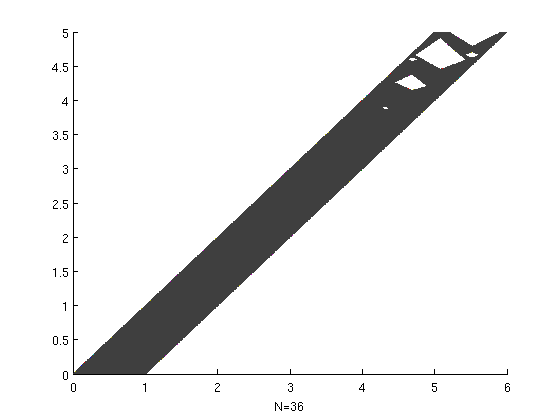}
\includegraphics[width=55mm]{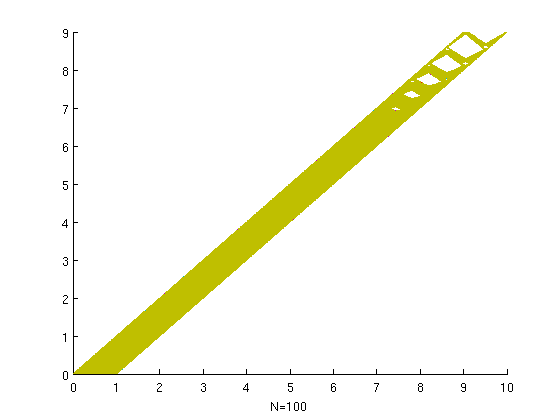}
\caption{Attractors plotted for several values of $N$.}\label{pic:atractors}
\end{figure}

\subsection{Entropy}\label{sec:entropy}

Entropy for a dynamical system has been introduced by Kolmogorov and Sinai in 1959. For a dynamical system related to continued fraction expansions it gives information about the speed of convergence for a typical point in that system.
Entropy as a function of a parameter $\alpha$ is widely studied for Nakada's $\alpha$-expansions (see~\cite{CMPT,CT12,CT13,LM,NN}). Also for the $(a,b)$-continued fractions the entropy is studied~\cite{CIT}.
A special interest goes to monotonicity which can be proven (for both cases) by using matching. It is shown that matching implies monotonicity and that matching holds almost everywhere.
In this section we look at entropy as a function of $\alpha\in(0,\sqrt{N}-1]$ for a fixed $N$. For our family it we do not know whether matching holds almost everywhere. In fact, it is not even clear whether matching occurs for all $N$. Also, if matching implies monotonicity is not clear. Certain conditions used to prove it in the case of Nakada $\alpha$-expansions do not hold for our family and therefore we cannot mimic the proof of~\cite{CT12}.
Let us first give the definition of matching.
\begin{definition}[Matching]\label{def:match}
 We say matching holds for $\alpha$ if there is an $K,M$ such that $T_{N,\alpha}^{K}(\alpha+1)=T_{N,\alpha}^M(\alpha)$. The numbers $K,M$ are called the matching exponents, $K-M$ is called the matching index and an interval $(c,d)$ such that for all $\alpha\in(c,d)$ we have the same matching exponents is called a matching interval.\end{definition}

\noindent Note that for Nakada's $\alpha$-expansions all rationals have a finite expansion and thus match in $0$ or before. For any rational in the domain a matching interval is found.
For our family all rationals will have an infinite expansion. Since for all $\alpha\in(0,\sqrt{N}-1)$ the expansion is infinite there are no values for which we find matching `trivially'.
Even if we find an $\alpha$ for which matching holds it is hard to conclude it must hold in a neighborhood of $\alpha$ as well since we cannot find an algebraic relation.
The presence of the holes for some $\alpha$ might also lead to problems.
The only reason to believe matching can help us to prove monotonicity is the fact that it works for related families and for specific choices of $N$ and $\alpha$ we can find matching where we can see the function is monotonic.
We will now discuss the $2_\alpha$-expansions in detail. Moreover, we proof that the entropy is constant for $\alpha \in \left( \frac{\sqrt{33}-5}{2}, \sqrt{2}-1\right)$ when $N=2$.  

\subsubsection{The entropy of $2_\alpha$-expansions}
We start with an example for which there is no $\alpha$ such that we have an attractor strictly smaller than the interval. Also, simulation indicates that there seems to be a plateau in the neighborhood of $\sqrt{2}-1$. For this value we can calculate the entropy since we have the density for this specific case of $\alpha$; see Section~\ref{sec:N2asqrt2min1} also see Figure~\ref{fig:entN2} for a plot of the entropy function.
\begin{figure}[h!]
\includegraphics[width=110mm]{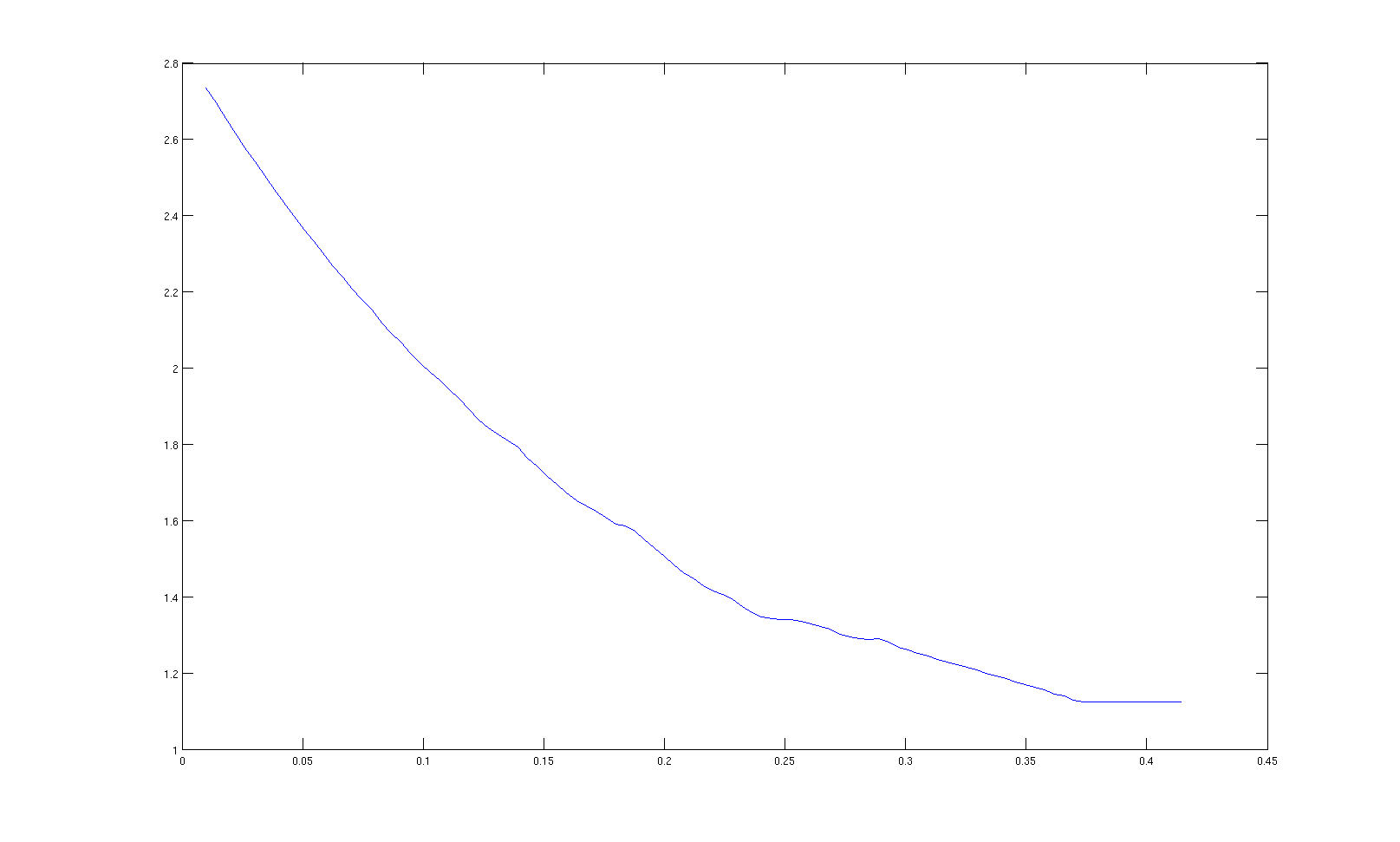}\caption{Entropy as function of $\alpha$ for $N=2$.}\label{fig:entN2}
\end{figure}
When taking a closer look at this plateau we found that on the interval $[\frac{\sqrt{33}-5}{2},\sqrt{2}-1]$ the entropy is constant on this interval. The point $\frac{\sqrt{33}-5}{2}$ is the point for which all smaller $\alpha$ have that there are always 5 or more branches and for all larger $\alpha$ there are always 4 branches.
If we look at a simulation of the natural extension it seems that for these values of $\alpha$ we can construct a natural extension. Indeed this turned out to be the case (see Theorem~\ref{thm:natexonint}).
For $\frac{\sqrt{33}-5}{2}$ we find matching exponents $(0,2)$ and for $\sqrt{2}-1$ we find matching exponents $(0,1)$. Inside the interval itself we find $(3,3)$. As the value of the entropy, These values were first found by simulation, in Theorem~\ref{thm:matching on last interval} we give a proof of this.

\begin{theorem}\label{thm:matching on last interval}
Let $N=2$, and let $\alpha \in \left( \frac{\sqrt{33}-5}{2}, \sqrt{2}-1\right)$. Then $T^3(\alpha ) = T^3(\alpha +1)$.
\end{theorem}

\begin{proof}
Note that the interval $(\alpha, \alpha +1)$ has as natural partition $\{ I_1,I_2,I_3,I_4\}$, where
$$
I_1= \Big( \frac{2}{\alpha +2},\alpha +1\Big) , \quad I_2= \Big( \frac{2}{\alpha+3},\frac{2}{\alpha+2}\Big] ,
\quad I_3= \Big( \frac{2}{\alpha+4},\frac{2}{\alpha+3}\Big] ,
$$
and
$$
I_4 = \Big[ \alpha, \frac{2}{\alpha+4}\Big] ,
$$
where
$$
T(x) = \frac{2}{x} - d,\quad \text{if $x\in I_d$\,\, for $d=1,2,3,4$}.
$$
An easy calculation shows that
$$
T(\alpha ) = \frac{2-4\alpha}{\alpha}\in I_1,
$$
(and $T(\alpha ) = \frac{2}{\alpha+2}$ when $\alpha = \sqrt{2}-1$, and $T(\alpha ) = \alpha +1$ when $\alpha=\frac{\sqrt{33}-5}{2}$), so that
$$
T^2(\alpha ) = \frac{3\alpha -1}{1-2\alpha}.
$$
Furthermore, we have that
$$
T(\alpha+1) = \frac{1-\alpha}{\alpha+1} \in I_4,
$$
(and $T(\alpha +1) = \sqrt{2}-1$ when $\alpha = \sqrt{2}-1$;  $T(\alpha +1) = \frac{2}{\alpha +4}$ when $\alpha = \frac{\sqrt{33}-5}{2}$), so that
$$
T^2(\alpha+1) = \frac{6\alpha -2}{1-\alpha}.
$$
Now let
$$
 K_1= \left(\frac{\sqrt{33}-5}{2},\frac{\sqrt{51}-6}{3}\right], \quad
 K_2= \left(\frac{\sqrt{51}-6}{3},\frac{\sqrt{129}-9}{6}\right] \\
$$
and
$$
K_3= \left(\frac{\sqrt{129}-9}{6},\sqrt{2}-1\right) \ .
$$
\newline For $\alpha\in K_1$ we have that $T^2(\alpha )\in I_3$ and so
$$
T^3(\alpha ) = \frac{5-13\alpha}{3\alpha -1}
$$
and $T^2(\alpha +1 )\in I_4$ which results in
$$
T^3(\alpha +1) = \frac{5-13\alpha}{3\alpha -1} = T^3(\alpha ) \ .
$$
For $\alpha\in K_2$ we have that  $T^2(\alpha )\in I_2$ and so
$$
T^3(\alpha ) = \frac{4-10\alpha}{3\alpha -1}
$$
and $T^2(\alpha+1)\in I_3$ which results in
$$
T^3(\alpha +1) = \frac{4-10\alpha}{3\alpha -1} = T^3(\alpha ) \ .
$$
For $\alpha\in K_3$ we have that  $T^2(\alpha )\in I_1$ and so
$$
T^3(\alpha) = \frac{3-7\alpha}{3\alpha -1}
$$
and $T^2(\alpha+1)\in I_2$ which results in
$$
T^3(\alpha+1) = \frac{3-7\alpha}{3\alpha -1} = T^3(\alpha ) \ .
$$
\end{proof}

\noindent Earlier we thought we were just lucky finding the natural extension in case $N=2$ and $\alpha = \sqrt{2}-1$. Note that from this natural extension we immediately also have the case $N=2$, $\alpha = \frac{\sqrt{33}-5}{2}$; just ``invert'' the time and exchange the two coordinates in the natural extension we found for $N=2$ and $\alpha = \sqrt{2}-1$. However, from Theorem~\ref{thm:matching on last interval} it immediately follows that we can also `build' the natural extension for every $\alpha \in \left( \frac{\sqrt{33}-5}{2},\sqrt{2}-1\right)$. Clearly, from Theorem~\ref{thm:matching on last interval} we see that we have three different cases.

\begin{theorem}\label{thm:natexonint}
 For $\alpha\in\left( \frac{\sqrt{33}-5}{2}, \sqrt{2}-1\right)$ the natural extension can be build as in Figure~\ref{fig:natexak1}. Moreover the invariant density is given by
 \begin{eqnarray*}
f(x)&=&H \large(\frac{D}{2+Dx}\textbf{1}_{(\alpha,T(\alpha+1))} +\frac{E}{2+Ex}\textbf{1}_{(T(\alpha+1),T^2(\alpha))}+\frac{F}{2+Fx}\textbf{1}_{(T^2(\alpha),\alpha+1)}\\
&-&\frac{A}{2+Ax}\textbf{1}_{(\alpha,T^2(\alpha+1))}-\frac{B}{2+Bx}\textbf{1}_{(T^2(\alpha+1),T(\alpha))}-\frac{C}{2+Cx}\textbf{1}_{(T(\alpha),\alpha+1)} \large) \\
\end{eqnarray*}
with $A=\frac{\sqrt{33}-5}{2}, B=\sqrt{2}-1, C=\frac{\sqrt{33}-3}{6},  D=2\sqrt{2}-2, E=\frac{\sqrt{33}-3}{2}, F=\sqrt{2}$ and $H^{-1}=\log\left(\frac{1}{32}(3+2\sqrt{2})(7+\sqrt{33})(\sqrt{33}-5)^2\right)\approx 0.25$ the normalizing constant.
\end{theorem}

\begin{proof}

We guessed the shape of the domain of natural extension by studying a simulation. For the map on this domain we used $\mathcal{T}(x,y)=\left(T(x),\frac{2}{d_1(x)+y}\right)$.

\begin{figure}[h]
\import{home/niels/natexak1.pic}{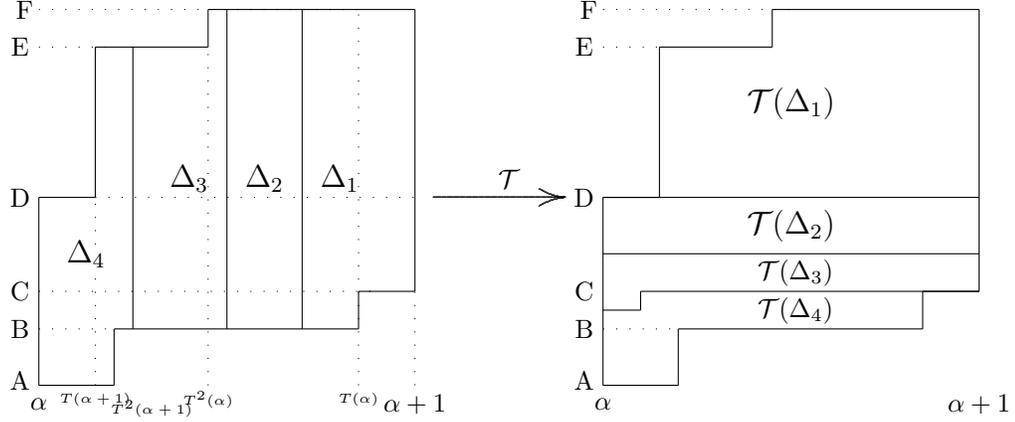}\caption{$\Omega$ and $\mathcal{T}(\Omega)$ with $\alpha\in K_1$.}\label{fig:natexak1}
\end{figure}
\noindent For $\alpha\in K_1$, we find the following equations:
$$
\begin{array}{l c l}
A=\frac{2}{4+E} & \quad & A=\frac{\sqrt{33}-5}{2}\\
B=\frac{2}{4+D} & \quad & B=\sqrt{2}-1\\
C=\frac{2}{3+E} & \quad & C=\frac{\sqrt{33}-3}{6}\\
D=\frac{2}{2+A} & \text{implying that} & D=2\sqrt{2}-2\\
E=\frac{2}{1+C} & \quad & E=\frac{\sqrt{33}-3}{2}\\
F=\frac{2}{1+B} & \quad & F=\sqrt{2} \ .\\
\end{array}
$$
A similar picture emerges for $\alpha\in K_2$ and $\alpha\in K_3$. Moreover, you will find the same set of equations and thus the same heights!
Note that for $\alpha<\frac{2}{5}$ we have $T^2(\alpha)<T(\alpha)$ for $\alpha=\frac{2}{5}$ we have $T^2(\alpha)=T(\alpha)$ and for $\alpha>\frac{2}{5}$ we have $T^2(\alpha)>T(\alpha)$. When you integrate over the second coordinate you find the density given in the statement of the theorem. For the normalizing constant we have the following integral

\begin{eqnarray*}
H&=&\int_{\alpha}^{\alpha+1} \frac{D}{2+Dx}\textbf{1}_{(\alpha,T(\alpha+1))} \ldots -\frac{C}{2+Cx}\textbf{1}_{(T(\alpha),\alpha+1)} \, dx\\
 &=& \log\left(\frac{2+D T(\alpha+1)}{2+D\alpha}\right)+\ldots+\log\left(\frac{2+CT(\alpha)}{2+C(\alpha+1)}\right)\\
 &=& \log\left(\frac{2+D T(\alpha+1)}{2+D\alpha}\cdots\frac{2+CT(\alpha)}{2+C(\alpha+1)}\right)\\
\end{eqnarray*}
It seems that $H$ depends on $\alpha$ but this is not the case. We will calculate one term since all others have a similar calculation.

\begin{eqnarray*}
\log\left(\frac{2+D T(\alpha+1)}{2+D\alpha}\right)&=&\log\left(\frac{2+(2\sqrt{2}-2)T(\alpha+1)}{2+(2\sqrt{2}-2)\alpha}\right)\\
&=&\log\left(\frac{2+(2\sqrt{2}-2)\frac{1-\alpha}{1+\alpha}}{2+(2\sqrt{2}-2)\alpha}\right)\\
&=&\log\left(\frac{2(1+\alpha)+(2\sqrt{2}-2)(1-\alpha)}{2+(2\sqrt{2}-2)\alpha}\right)+\log(1+\alpha)\\
&=&\log\left(\sqrt{2}\frac{\sqrt{2}(1+\alpha)+\sqrt{2}(\sqrt{2}-1)(1-\alpha)}{2+(2\sqrt{2}-2)\alpha}\right)+\log(1+\alpha)\\
&=&\log\left(\sqrt{2}\frac{2+(2\sqrt{2}-2)\alpha}{2+(2\sqrt{2}-2)\alpha}\right)+\log(1+\alpha)\\
&=&\log(\sqrt{2})+\log(1+\alpha)\\
\end{eqnarray*}
When calculating the second term one finds $\frac{1}{4}(\sqrt{33}-5)-\log(1+\alpha)$ and so this cancels the $\log(1+\alpha)$ term.
\end{proof}
\noindent One might hope that when calculating the entropy using Rohlin's formula, terms will cancel as well. These integrals result in $Li_2$ functions depending on $\alpha$ and things are not so easy anymore.
We provide a more elegant proof to show that the entropy is constant on $\left( \frac{\sqrt{33}-5}{2}, \sqrt{2}-1\right)$ and then calculate the entropy for $\alpha=\sqrt{2}-1$. We will use quilting introduced in~\cite{KSS}. Proposition 1 in this article can be formulated (specific to our case) in the following way:
\begin{proposition}\label{prop:quilt}
Let $(\mathcal{T}_{\alpha},\Omega_{\alpha},\mathcal{B}_{\alpha},\mu)$ and $(\mathcal{T}_{\beta},\Omega_{\beta},\mathcal{B}_{\beta},\mu)$ be two dynamical systems as in our setting. Furthermore let $D_1=\Omega_{\alpha}\backslash\Omega_{\beta}$ and $A_1=\Omega_{\alpha}\backslash\Omega_{\beta}$.
If there is a $k\in\mathbb{N}$ such that $\mathcal{T}_{\alpha}^k(D_1)= \mathcal{T}_{\beta}^k(A_1)$ then the dynamical systems are isomorphic.
\end{proposition}
\noindent Since isomorphic systems have the same entropy it will give us the following corollary.
\begin{corollary}
For $N=2$ the entropy function is constant on $\left( \frac{\sqrt{33}-5}{2}, \sqrt{2}-1\right)$ and the value is approximately $1.14$.
\end{corollary}

\begin{proof}
We show that for $k=3$ we satisfy the condition in Proposition~\ref{prop:quilt}. Define $D_i=\mathcal{T}_{\alpha}^{i-1}(D_1)$ and $A_i=\mathcal{T}_{\beta}^{i-1}(A_1)$ for $k=1,2,3,4$.
We find the following regions (see Figure~\ref{fig:quilt}):
\begin{eqnarray*}
 D_1 &=& [\alpha,\beta]\times[A,D],\\
 D_2 &=& [T(\beta),T(\alpha)]\times[B,C],\\
 D_3 &=& [T^2(\beta),T^2(\alpha)]\times[E,F],\\
 D_4 &=& [T^3(\beta),T^3(\alpha)]\times\left[\frac{2}{3+F},\frac{2}{3+E}\right],\\
 A_1 &=& [\alpha+1,\beta+1]\times[C,F],\\
 A_2 &=& [T(\beta+1),T(\alpha+1)]\times[D,E],\\
 A_3 &=& [T^2(\beta+1),T^2(\alpha+1)]\times[A,B],\\
 A_4 &=& [T^3(\beta+1),T^3(\alpha+1)]\times\left[\frac{2}{4+B},\frac{2}{4+A}\right].\\ \end{eqnarray*}
Note that since we have matching $[T^3(\beta),T^3(\alpha)]=[T^3(\beta+1),T^3(\alpha+1)]$. Now
$$
\frac{2}{3+F}=\frac{2}{3+\sqrt{2}}= \frac{2}{4+B}=\frac{2}{4+\sqrt{2}-1} \ ,
$$
$$
\frac{2}{3+E}=\frac{2}{3+\frac{\sqrt{33}-3}{2}}= \frac{2}{4+A}=\frac{2}{4+\frac{\sqrt{33}-5}{2}}
$$
and so we find $D_4=A_4$.
\begin{figure}[h]
\import{home/niels/natexquilt.pic}{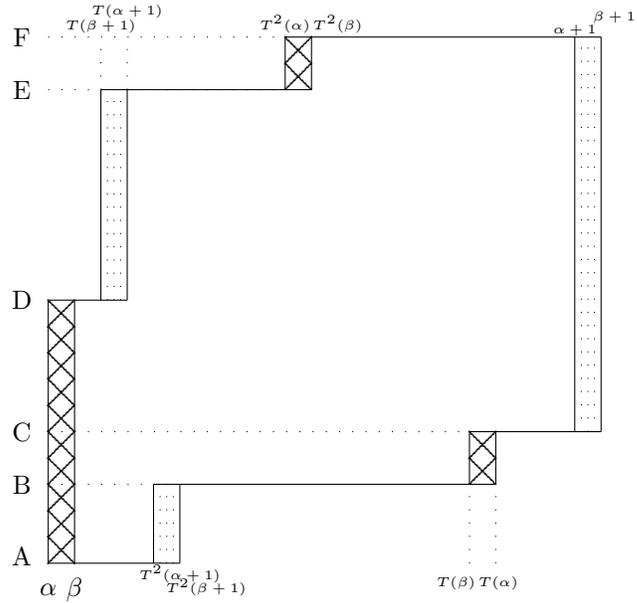}\caption{Illustration of the quilting.}\label{fig:quilt}
\end{figure}

\noindent For the value of the entropy we use Rohlin's formula for $\alpha=\sqrt{2}-1$ (see~\cite{DK1,S});
\begin{eqnarray*}
h(T_{\sqrt{2}-1})&=&\int_{\sqrt{2}-1}^{\sqrt{2}}\log|T^{\prime}(x)|f(x)\, dx\\
&=&H\int_{\sqrt{2}-1}^{\sqrt{2}}\left(\log(2)-2\log(x)\right)f(x)\, dx\\
&=&\log(2)-2H\int_{\sqrt{2}-1}^{\sqrt{2}}\log(x)f(x)\, dx\\
&=&\log(2)-2H\int_{\sqrt{2}-1}^{\sqrt{2}}\log(x)\large( (\frac{\sqrt{33}-3}{4+(\sqrt{33}-3)x}-\frac{\sqrt{33}-5}{4+(\sqrt{33}-5)x})1_{\sqrt{2}-1,2(\sqrt{2}-1)}\\
& & \, +(\frac{\sqrt{33}-3}{4+(\sqrt{33}-3)x}-\frac{\sqrt{33}-3}{12+(\sqrt{33}-3)x})1_{2(\sqrt{2}-1),\sqrt{2}}\large)\, dx\\
&=&\log(2)-2H\large( (Li_2(-\frac{x(\sqrt{33}-3)}{4})+\log(\frac{x(\sqrt{33}-3)}{4}+1)\\
& &\, -Li_2(-\frac{x(\sqrt{33}-5)}{4})+\log(\frac{x(\sqrt{33}-5)}{4}+1))\large|_{\sqrt{2}-1}^{2(\sqrt{2}-1)}\\
& & \, + (Li_2(-\frac{x(\sqrt{33}-3)}{4})+\log(\frac{x(\sqrt{33}-3)}{4}+1)\\
& &\, -Li_2(-\frac{x(\sqrt{33}-5)}{12})+\log(\frac{x(\sqrt{33}-5)}{12}+1))\large|_{2(\sqrt{2}-1)}^{\sqrt{2}}\large)\\
&\approx&1.14 \ .
\end{eqnarray*}
\end{proof}
\noindent By looking at the graph displayed in Figure~\ref{fig:entN2} we can't really find other matching exponents easily. To check for other matching exponents we can do the following.  Suppose we are interested in finding a matching interval with exponents $(n_1,n_2)$.
 We select a large number random points (say $10\,000$) from $(0,\sqrt{N}-1)$. Then we looked at $T_{\alpha}^{n_1}(\alpha)-T_{\alpha}^{n_2}(\alpha+1)$ for these random points and we checked to see whether it is very close to 0. Note that if an interval was found this way with matching exponents $(n_1,n_2)$ then we also find that interval for $(n_1+1,n_2+1)$.
Table \ref{tab:cases2} shows which matching exponents we found.
\begin{table}[h!]
\centering
\begin{tabular}{ c |c | c | c |c|c|c|c|c|c|c}
 M\textbackslash K & 1 & 2 & 3 & 4 & 5 & 6 & 7 & 8 & 9 & 10 	\\
\hline
1 & 0 & 0 & 1 & 0 & 1 & 0 & 0 & 0 & 0 & 0 \\
2 & 0 & 0 & 0 & 1 & 0 & 1 & 0 & 0 & 0 & 0 \\
3 & 0 & 0 & 1 & 0 & 1 & 0 & 1 & 0 & 1 & 0 \\
4 & 0 & 0 & 0 & 1 & 0 & 1 & 0 & 1 & 0 & 1 \\
5 & 0 & 0 & 0 & 0 & 1 & 0 & 1 & 0 & 1 & 0 \\
6 & 0 & 0 & 0 & 0 & 0 & 1 & 0 & 1 & 0 & 1 \\
7 & 0 & 0 & 0 & 0 & 0 & 0 & 1 & 0 & 1 & 0 \\
8 & 0 & 0 & 0 & 0 & 0 & 0 & 0 & 1 & 0 & 1 \\
9 & 0 & 0 & 0 & 0 & 0 & 0 & 0 & 0 & 1 & 0 \\
10 & 0 & 0 & 0 & 0 & 0 & 0 & 0 & 1 & 0 & 1 \\
\end{tabular}
\caption{observed matching exponents for $N=2$: 1 if seen, 0 if not.}\label{tab:cases2}
\end{table}
This is very different from Nakada's $\alpha$-continued fractions where you can find all possible matching exponents. The fact that we did not observe them does not mean they are not there. Maybe they are too small to observe using this method.

\subsubsection{The entropy of $36_\alpha$-expansions}
For $N\geq 9$ we expect different behavior due to the fact that for some $\alpha$ there is at least one subinterval on which the invariant measure is zero.
If we pick $N=36$ we have a map with only full branches for $\alpha=1,2,3$.
Figure~\ref{fig:ent36} shows the entropy as function of $\alpha$. The stars indicate those values which we could calculate theoretically.
\begin{figure}[h!]
 \includegraphics[width=110mm]{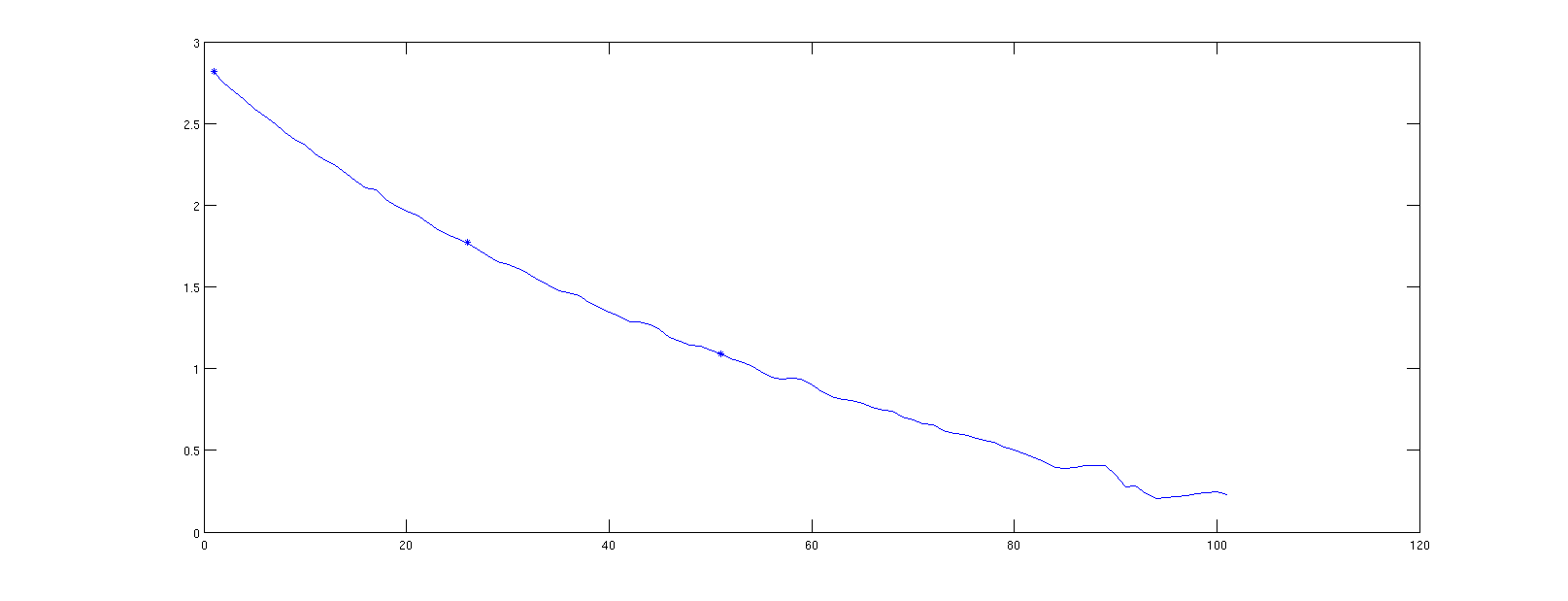}\caption{Entropy as function of $a$ for $N=36$.}\label{fig:ent36}
\end{figure}

\noindent Clearly, we can observe plateaus but if we look at the matching exponents we can observe that for all $M,K\leq 10$ the only matching exponents we find are $(3,3),(4,4),\ldots(10,10)$.

\section{Conclusion}
We have seen that the general form of the examples given yields a rather large family. In some examples we were able to construct the natural extension and therefore to find the invariant measure. In other examples this was not the case. There does not seem to be an easy rule which tells you when the method will work and when it does not.
The subfamily of the $N$-expansions we studied is not new, but it has not been studied in this detail with finitely many digits. Note that having the Gauss-Kuzmin-L\'evy method for approximating the densities allowed us to study the entropy much easier due to computation time. We have seen that matching is helpful to prove monotonicity even though we did not mimic the proof for $\alpha$-expansions.
Motivated by similar results in the case of Nakada's $\alpha$-expansions the  following questions about entropy arise:
\begin{itemize}[noitemsep]
 \item For every $N\in\mathbb{N}_{\geq 2}$ is there an interval in $(0,\sqrt{N}-1)$ for which the entropy function is constant?
 \item For a fixed $N \in\mathbb{N}_{\geq 2}$ for which $\alpha\in (0,\sqrt{N}-1)$ do we have matching?
 \item Does matching holds on an open dense set? Does matching hold almost everywhere?
 \item What is the influence of an attractor strictly smaller than the interval $[\alpha, \alpha+1]$ on the entropy?
\end{itemize}

\bibliography{mybib}

\begin{thebibliography}{10}
\expandafter\ifx\csname url\endcsname\relax
  \def\url#1{\texttt{#1}}\fi
\expandafter\ifx\csname urlprefix\endcsname\relax\def\urlprefix{URL }\fi
\expandafter\ifx\csname href\endcsname\relax
  \def\href#1#2{#2} \def\path#1{#1}\fi

\bibitem{SW}
F.~Schweiger, M.~Waterman, Some remarks on {K}uzmin's theorem for
  {$F$}-expansions, J. Number Theory 5 (1973) 123--131.

\bibitem{DK1}
K.~Dajani, C.~Kraaikamp, Ergodic theory of numbers, Vol.~29 of Carus
  Mathematical Monographs, Mathematical Association of America, Washington, DC,
  2002.

\bibitem{IK}
M.~Iosifescu, C.~Kraaikamp, Metrical theory of continued fractions, Vol. 547 of
  Mathematics and its Applications, Kluwer Academic Publishers, Dordrecht,
  2002.

\bibitem{S}
F.~Schweiger, Continued Fractions and their Generalizations: A short history of
  $f$-expansions, docentpress, 2016.

\bibitem{Le}
J.~Lehner, Semiregular continued fractions whose partial denominators are {$1$}
  or {$2$}, in: The mathematical legacy of {W}ilhelm {M}agnus: groups, geometry
  and special functions ({B}rooklyn, {NY}, 1992), Vol. 169 of Contemp. Math.,
  Amer. Math. Soc., Providence, RI, 1994, pp. 407--410.

\bibitem{DK2}
K.~Dajani, C.~Kraaikamp, ``{T}he mother of all continued fractions'', Colloq.
  Math. 84/85~(part 1) (2000) 109--123, dedicated to the memory of Anzelm
  Iwanik.

\bibitem{DKW}
K.~Dajani, C.~Kraaikamp, N.~van~der Wekken, Ergodicity of {$N$}-continued
  fraction expansions, J. Number Theory 133~(9) (2013) 3183--3204.

\bibitem{AW}
M.~Anselm, S.~H. Weintraub, A generalization of continued fractions, J. Number
  Theory 131~(12) (2011) 2442--2460.

\bibitem{BGK}
E.~B. Burger, J.~Gell-Redman, R.~Kravitz, D.~Walton, N.~Yates, Shrinking the
  period lengths of continued fractions while still capturing convergents, J.
  Number Theory 128~(1) (2008) 144--153.

\bibitem{L14}
N.~D.~S. Langeveld, \href{repository.tudelft.nl}{Finding infintely many even or
  odd continued fractions by introducing a new family of maps}, Master's thesis
  (2015).
\newline\urlprefix\url{repository.tudelft.nl}

\bibitem{DHMK}
K.~Dajani, D.~Hensley, C.~Kraaikamp, V.~Masarotto, Arithmetic and ergodic
  properties of `flipped' continued fraction algorithms, Acta Arith. 153~(1)
  (2012) 51--79.

\bibitem{C1}
G.~H. Choe, Generalized continued fractions, Appl. Math. Comput. 109~(2-3)
  (2000) 287--299.

\bibitem{La}
D.~Lascu, Dependence with complete connections and the gauss-kuzmin theorem for
  $n$-continued fractions (2015).

\bibitem{C3}
G.~H. Choe, Computational ergodic theory, Vol.~13 of Algorithms and Computation
  in Mathematics, Springer-Verlag, Berlin, 2005.

\bibitem{C2}
G.~H. Choe, C.~Kim, The {K}hintchine constants for generalized continued
  fractions, Appl. Math. Comput. 144~(2-3) (2003) 397--411.

\bibitem{BCIT}
C.~Bonanno, C.~Carminati, S.~Isola, G.~Tiozzo, Dynamics of continued fractions
  and kneading sequences of unimodal maps, Discrete Contin. Dyn. Syst. 33~(4)
  (2013) 1313--1332.

\bibitem{CMPT}
C.~Carminati, S.~Marmi, A.~Profeti, G.~Tiozzo, The entropy of
  {$\alpha$}-continued fractions: numerical results, Nonlinearity 23~(10)
  (2010) 2429--2456.

\bibitem{CT12}
C.~Carminati, G.~Tiozzo, A canonical thickening of {$\Bbb Q$} and the entropy
  of {$\alpha$}-continued fraction transformations, Ergodic Theory Dynam.
  Systems 32~(4) (2012) 1249--1269.

\bibitem{CT13}
C.~Carminati, G.~Tiozzo, Tuning and plateaux for the entropy of
  {$\alpha$}-continued fractions, Nonlinearity 26~(4) (2013) 1049--1070.

\bibitem{KSS2}
C.~Kraaikamp, T.~A. Schmidt, W.~Steiner, Natural extensions and entropy of
  {$\alpha$}-continued fractions, Nonlinearity 25~(8) (2012) 2207--2243.

\bibitem{LM}
L.~Luzzi, S.~Marmi, On the entropy of {J}apanese continued fractions, Discrete
  Contin. Dyn. Syst. 20~(3) (2008) 673--711.

\bibitem{NN}
H.~Nakada, R.~Natsui, The non-monotonicity of the entropy of
  {$\alpha$}-continued fraction transformations, Nonlinearity 21~(6) (2008)
  1207--1225.

\bibitem{KU0}
S.~Katok, I.~Ugarcovici, Theory of {$(a,b)$}-continued fraction transformations
  and applications, Electron. Res. Announc. Math. Sci. 17 (2010) 20--33.

\bibitem{KU1}
S.~Katok, I.~Ugarcovici, Structure of attractors for {$(a,b)$}-continued
  fraction transformations, J. Mod. Dyn. 4~(4) (2010) 637--691.

\bibitem{KU2}
S.~Katok, I.~Ugarcovici, Applications of {$(a,b)$}-continued fraction
  transformations, Ergodic Theory Dynam. Systems 32~(2) (2012) 755--777.

\bibitem{CIT}
C.~Carminati, S.~Isola, G.~Tiozzo, Continued fractions with
  $sl(2,\mathbb{Z})$-branches: combinatorics and entropy (2013).

\bibitem{K}
A.~Khintchine, Kettenbr\"uche, B. G. Teubner Verlagsgesellschaft, Leipzig,
  1956.

\bibitem{DKL}
K.~Dajani, C.~Kraaikamp, N.~D.~S. Langeveld, Continued fraction expansions with
  variable numerators, Ramanujan J. 37~(3) (2015) 617--639.

\bibitem{KSS}
C.~Kraaikamp, T.~A. Schmidt, I.~Smeets, Natural extensions for
  {$\alpha$}-{R}osen continued fractions, J. Math. Soc. Japan 62~(2) (2010)
  649--671.

\end{thebibliography}
\bibliographystyle{elsarticle-num}
 
\end{document}